\documentclass{amsart}[12pt]

\let\oldmarginpar\marginpar
\renewcommand\marginpar[1]{\-\oldmarginpar[\raggedleft\footnotesize #1]%
{\raggedright\footnotesize #1}}

\usepackage{xypic}
\usepackage{eucal}

\usepackage{amsmath}
\usepackage{amstext}
\usepackage{amssymb}
\usepackage{amsthm}
\usepackage{amscd}

\usepackage{lscape}
\usepackage{longtable}

\usepackage{epsfig}
\input txdtools
\let\et=\etexdraw
\def\etexdraw{\drawbb\et}

\RequirePackage[dvipsnames,usenames]{color}

\theoremstyle{plain}
\newtheorem{thm}{Theorem}[section]
\newtheorem{thm*}{Theorem}
\newtheorem{lem}[thm]{Lemma}
\newtheorem{prop}[thm]{Proposition}
\newtheorem{prop*}[thm*]{Proposition}
\newtheorem{cor}[thm]{Corollary}

\theoremstyle{definition}

\theoremstyle{remark}

\DeclareMathOperator{\End}{End}

\DeclareMathOperator{\Jac}{Jac}

\begin{document}

\title
[An upper bound on the number of F-jumping coefficients of a principal ideal]
{An upper bound on the number of F-jumping coefficients of a principal ideal}

\author{Mordechai Katzman}
\address{Department of Pure Mathematics,
University of Sheffield, Hicks Building, Sheffield S3 7RH, United Kingdom}
\email{M.Katzman@sheffield.ac.uk}

\author{Gennady Lyubeznik}
\address{Department of Mathematics, University of Minnesota, Minneapolis, MN, 55455, USA}
\email{gennady@math.umn.edu}
\author{Wenliang Zhang}
\address{Department of Mathematics, University of Michigan, Ann Arbor, MI, 48109, USA}
\email{wlzhang@umich.edu}


\date{}


\begin{abstract}
We prove a result relating the Jacobian ideal and the generalized test ideal associated to a principal ideal in $R=k[x_1,\dots,x_n]$ with 
$[k:k^p]<\infty$ or in $R=k[[x_1,\dots,x_n]]$ with an arbitrary field $k$ of characteristic $p>0$. 
As a consequence of this result, we establish an upper bound on the number of $F$-jumping coefficients of a principal ideal with an isolated singularity.  
\end{abstract}

\maketitle

\numberwithin{equation}{thm}
\section{Introduction}

In characteristic 0, one can define invariants,  called {\it jumping coefficients} in \cite{ELSV04}, attached to an ideal sheaf on a smooth variety via multiplier ideals. 
These jumping coefficients consist of an ascending chain of positive rational numbers and they encode interesting geometric and algebraic information 
(see \cite[Chapter 9]{Laz04} for details). In \cite{ELSV04} the following connection between Jacobian and multiplier ideals is discovered.

\begin{thm}[Proposition 3.8\footnote{In the statement of \cite[Proposition 3.8]{ELSV04}, $f$ is assumed to have an isolated singularity. 
However, this assumption is not needed in the proof of the statement. } in \cite{ELSV04}]
\label{Jacobian-multiplier-ideals}
Given $f\in \mathbb{C}[x_1,\dots,x_n]$, one has that
\[\Jac(f)\subseteq \mathcal{J}((f)^{1-\epsilon}),\text{ for all }\epsilon>0,\]
where $\Jac(f)=(f,\frac{\partial f}{\partial x_1},\dots,\frac{\partial f}{\partial x_n})$ and $\mathcal{J}((f)^{1-\epsilon})$
is the multiplier ideal of the pair $(\mathbb{C}[x_1,\dots,x_n],(f)^{1-\epsilon})$.
\end{thm}

As a consequence, one has 

\begin{cor}
If $f\in R=\mathbb{C}[x_1,\dots,x_n]$ has an isolated singularity, then $f$ has at most $\dim_{\mathbb{C}}(\frac{R}{\Jac(f)})+1$ jumping coefficients in $[0,1]$.
\end{cor}

The purpose of this paper is to extend these results to characteristic $p>0$.\par

In characteristic $p>0$, Hara and Yoshida introduced, in \cite{HY03}, an analogue of the multiplier ideals, the generalized test ideals. 
Generalized test ideals can defined in any noetherian ring of characteristic $p>0$, but the definition is less technical when the ring $R$ is an $F$-finite regular ring
or an excellent regular local ring. When $R$ is an $F$-finite regular ring (see \cite{BMS08}) or an excellent regular local ring (see \cite{KLZ09}),
for each ideal $J$ of $R$ and each positive integer $e$, there exists a unique smallest ideal $I_e(J)$ such that $J\subseteq (I_e(J))^{[p^e]}$. 
Then, for each nonnegative real number $t$, the generalized test ideal, $\tau(J^t)$, can be defined as 
\[\bigcup_{e}I_e(J^{\lceil tp^e\rceil}).\] 
In this context, we say that $c$ is an {\it $F$-jumping coefficient} of $J$ if $\tau(J^c)\subsetneqq \tau(J^{c'})$ for all $c'<c$. 
It is proved, in \cite{BMS09} when $R$ is an $F$-finite regular ring and in \cite{KLZ09} when $R$ is an excellent regular local ring,
that $F$-jumping coefficients of each principal ideal of $R$ consist of an ascending chain of positive rational numbers.

Let $R=k[x_1,\dots,x_n]$ with $[k:k^p]<\infty$ or $R=k[[x_1,\dots,x_n]]$ with $k$ an arbitrary field of characteristic $p>0$. 
Let $\frac{\partial f}{\partial x_i}$ denote the partial derivative of $f\in R$ with respect to $x_i$ and let 
$\Jac(f)=(f,\frac{\partial f}{\partial x_1},\dots,\frac{\partial f}{\partial x_n})$. 
Our main theorem of this paper is the following analogue of Theorem \ref{Jacobian-multiplier-ideals}.

\begin{thm}[Main Theorem]
Let $R$ be as above, then
\[\Jac(f)\subseteq \tau((f)^{1-\epsilon}),\]
for all $\epsilon>0$.
\end{thm}

When $k$ is a perfect field of characteristic $p>0$, then the singular locus of $R/(f)$ is determined by $\Jac(f)$. 
In particular, when $f$ (or equivalently $R/(f)$) has an isolated singularity, $\dim_k(\frac{R}{\Jac(f)})$ is finite. 
Note that in this case there are at most $\dim_k(\frac{R}{\Jac(f)})$ different ideals  between $R$ and $\Jac(f)$. Therefore, as a consequence of our Main Theorem, we have 
\begin{cor}
Let $k$ be a perfect field of characteristic $p>0$, and let $R=k[x_1,\dots,x_n]$ or $k[[x_1,\dots,x_n]]$. 
Then, for each $f\in R$ with isolated singularity, there are at most $\dim_k(\frac{R}{\Jac(f)})+1$ $F$-jumping coefficients of $f$ in $[0,1]$.
\end{cor}

\section{A result on differential operators}
In this section we consider differential operators over $R=\mathbb{Z}[x_1,\dots,x_n]$ (or $R=\mathbb{Z}[[x_1,\dots,x_n]]$, respectively): 
let $D_{m,i}:R\to R$ be the $\mathbb{Z}[x_1,\dots,x_{i-1},x_{i+1},\dots,x_n]$-linear
(or $\mathbb{Z}[[x_1,\dots,x_{i-1},x_{i+1},\dots,x_n]]$-linear, respectively) map that sends 
$x^{\ell}_i$ to $\binom{\ell}{m}x_i^{\ell-m}$ and let $D_{0,i}$ be the identity map. We may write $D_{m,i}$ as
\[D_{m,i}=\frac{1}{m!}\frac{\partial^m}{\partial x^m_i}:R\to R.\]

Even though the following proposition is stated over a field $k$ in \cite{Lyu10}, the same proof works over $\mathbb{Z}$.

\begin{prop}[Proposition 2.1 in \cite{Lyu10}]
\label{higher-derivatives}
For each $f\in R$, we have 
\[D_{m,i}\cdot f=\sum^m_{\ell=0}D_{\ell,i}(f)\cdot D_{m-\ell,i}\ {\rm in\ }\End_{\mathbb{Z}}(R),\]
i.e., given $f,g\in R$, we have
\[D_{m,i}( fg)=\sum^m_{\ell=0}D_{\ell,i}(f)D_{m-\ell,i}(g)\ {\rm in\ } R.\]
\end{prop}

Our main result of this section is the following identity.

\begin{thm}
Given any $f\in R$ and a positive integer $m$, we have
\begin{equation}
\label{key-formula}
\sum_{\ell=0}^m(\ell-1)D_{\ell,i}(f)D_{m-\ell,i}(f^{m-1})=0.
\end{equation}
\end{thm}
\begin{proof}[Proof]
Writing $D_{m,i}$ as $\frac{1}{m!}\frac{\partial^m}{\partial x^m_i}$ and multiplying (\ref{key-formula}) by $m!$, we can rewrite (\ref{key-formula}) as
\[\sum^m_{\ell=0}(\ell-1)\binom{m}{\ell}\frac{\partial^{\ell}f}{\partial x_1^{\ell}}\frac{\partial^{m-\ell}f^{m-1}}{\partial x^{m-\ell}_i}=0.\]
To ease our notation, we will write $D_{\ell,i}(f)$ as $\partial^{\ell}(f)$. Hence the above equation becomes
\begin{equation}
\label{simplified-key-equation}
\sum^m_{\ell=0}(\ell-1)\binom{m}{\ell}\partial^{\ell}(f)\partial^{m-\ell}(f^{m-1})=0.
\end{equation}

We first expand
\begin{equation}\label{eqn1}
\sum^m_{\ell=0}(\ell-1)\binom{m}{\ell}\partial^{\ell}(f)\partial^{m-\ell}(f^{m-1})
=
-f\partial^m(f^{m-1})-
\sum^m_{\ell=2}\binom{m}{\ell}\partial^{\ell}(f)\partial^{m-\ell}(f^{m-1})+
\sum^m_{\ell=2}\ell\binom{m}{\ell}\partial^{\ell}(f)\partial^{m-\ell}(f^{m-1}) .
\end{equation}
Use the fact that 
$\ell\binom{m}{\ell}=m\binom{m-1}{\ell-1}$ to rewrite (\ref{eqn1}) as
\begin{equation}\label{eqn2}
-f\partial^m(f^{m-1})-
\sum^m_{\ell=2}\binom{m}{\ell}\partial^{\ell}(f)\partial^{m-\ell}(f^{m-1})+
\sum^m_{\ell=2}m\binom{m-1}{\ell-1}\partial^{\ell}(f)\partial^{m-\ell}(f^{m-1}) .
\end{equation}
Note that Leibniz's rule implies that 
$\partial^m(f^m)= \partial^m(f f^{m-1}) = \sum_{l=0}^m \binom{m}{l} \partial^l(f) \partial^{m-l}(f^{m-1})$ and rewrite (\ref{eqn2}) as
\begin{equation}\label{eqn3}
-f\partial^m(f^{m-1})-
\left(\partial^m(f^m)-f\partial^m(f^{m-1})-m\partial(f)\partial^{m-1}(f^{m-1})\right)+
\sum^{m-1}_{\ell=1}m\binom{m-1}{\ell}\partial^{\ell}(\partial(f))\partial^{m-1-\ell}(f^{m-1}) .
\end{equation}
Now 
\begin{eqnarray*}
m\partial(f)\partial^{m-1}(f^{m-1})+
\sum^{m-1}_{\ell=1} m \binom{m-1}{\ell}\partial^{\ell}(\partial(f))\partial^{m-1-\ell}(f^{m-1}) & = & 
m \sum^{m-1}_{\ell=0} \binom{m-1}{\ell}\partial^{\ell}(\partial(f))\partial^{m-1-\ell}(f^{m-1})\\
& = & \partial^{m-1} \left( m (\partial f) f^{m-1} \right) \\
& = & \partial^{m-1} \partial f^m=\partial^{m} f^m
\end{eqnarray*}
and (\ref{eqn3}) simplifies to
\begin{equation*}
-f\partial^m(f^{m-1})- \partial^m(f^m) + f\partial^m(f^{m-1}) + \partial^{m} f^m=0 .
\end{equation*}

\end{proof}

\section{Proof of the Main Theorem}
Throughout this section $R$ is either $k[x_1,\dots,x_n]$ with $[k:k^p]<\infty$ or $k[[x_1,\dots,x_n]]$ with $k$ an arbitrary field of characteristic $p>0$. 
In either case, by reducing the operators $D_{m,i}$ mod $p$, we get differential operators over $k$, which will still be denoted  $D_{m,i}$. 
Note that the identity (\ref{key-formula}) also holds over $k$ and that each differential operator $D_{m,i}$ is $R^{p^e}$-linear when $m<p^e$.

We begin with an easy observation.

\begin{lem}
\label{p^e-derivative}
Given any $f\in R$, we have
\[D_{p^e,i}(f^{p^e})=(\frac{\partial f}{\partial x_i})^{p^e}.\]
\end{lem}
\begin{proof}[Proof]
Since $D_{p^e,i}$ is $k$-linear and $f$ is a $k$-linear combination of monomials, it suffices to consider the case when $f$ is a monomial. 
Also since $D_{p^e,i}$ is $k[x_1,\dots,\hat{x}_i,\dots,x_n]$-linear (or $k[[x_1,\dots,\hat{x}_i,\dots,x_n]]$-linear), 
it suffices to consider that case when $f=x^t_i$ for some $t\in\mathbb{N}$. 
We have
\[D_{p^e,i}(f^{p^e})=\binom{tp^e}{p^e}x^{(t-1)p^e},\ {\rm and\ }(\frac{\partial f}{\partial x_i})^{p^e}=t^{p^e}x^{(t-1)p^e}.\]
But it is well-known that $\binom{tp^e}{p^e}\equiv t\ ({\rm mod}\ p)$ by Lucas' Theorem and $t^{p^e}\equiv t\ ({\rm mod}\ p)$ by Fermat's Little Theorem. 
\end{proof}

We can now prove our main theorem.
\begin{proof}[Proof of Main Theorem]
It is clear that it is enough to prove that
\[\Jac(f)\subseteq \tau((f)^{1-\frac{1}{p^e}})\]
for all integers $e>0$. Since $ \tau((f)^{1-\frac{1}{p^e}})=I_e(f^{p^e-1})$, the smallest ideal $J$ such that $f^{p^e-1}\subseteq J^{[p^e]}$, it suffices to prove that, for any ideal $J$ with $f^{p^e-1}\subseteq J^{[p^e]}$, we have
\[\frac{\partial f}{\partial x_i}\in J.\]
Since $D_{t,i}$ with $t<p^e$ is $R^{p^e}$-linear, we have
\[D_{t,i}(f^{p^e-1})\in J^{[p^e]},\ {\rm for\ all}\ t<p^e.\]
Setting $m=p^e$ in the identity (\ref{key-formula}), we have that
\[fD_{p^e,i}(f^{p^e-1})=\sum^{p^e}_{\ell=2}(\ell-1)D_{\ell,i}(f)D_{p^e-\ell,i}(f^{p^e-1})\in J^{[p^e]}.\]
According to Proposition \ref{higher-derivatives},
\[D_{p^e,i}(f^{p^e})=D_{p^e,i}(ff^{p^e-1})=\sum^{p^e}_{\ell=0}D_{\ell,i}(f)D_{p^e-\ell,i}(f^{p^e-1}),\]
and hence $D_{p^e,i}(f^{p^e})\in J^{[p^e]}$. Combining this with Lemma \ref{p^e-derivative}, we see that
\[(\frac{\partial f}{\partial x_i})^{p^e}\in J^{[p^e]}\]
and consequently
\[\frac{\partial f}{\partial x_i}\in J\]
since $R$ is regular. This finishes the proof of our main theorem.
\end{proof}

\subsection*{Acknowledgments.} The results in this paper were obtained while the first and third authors enjoyed the hospitality of the School of Mathematics at the University of Minnesota. This work was initiated at the Commutative Algebra MRC held in June 2010; the third author wishes to thank the organizers for providing a simulating atmosphere.

\end{document}